\documentclass[amstex,12pt, leqno]{amsart}
\usepackage{amsthm,amsmath,amsfonts,amssymb,enumerate,graphicx,hyperref,multicol,color,textcomp,hyperref,graphics}
\usepackage{bm,ifsym,float}
\newtheorem{lemma}{Lemma}
\newtheorem{theorem}{Theorem}
\newtheorem{remark}{Remark}

\numberwithin{equation}{section}
\newtheorem{thm}{Theorem}

\newtheorem{corollary}{Corollary}

{\qed\bigskip}

{\thanks
{
\begin{footnotesize}
\hspace*{-7mm}
\end{footnotesize}
}}

\begin{document}

\leftline{ \scriptsize \it  }
\title [Rate of approximation by new variants of Bernstein-Durrmeyer operators]
{Rate of approximation by new variants of Bernstein-Durrmeyer operators}

\maketitle
\begin{center}Asha Ram
Gairola\\
 Department of Mathematics,  Doon University\\
  Dehradun-248001 (Uttarakhand), India\\
  Email: ashagairola@gmail.com\\
  and\\
   Karunesh Kumar Singh\\
  Department of Applied Sciences and Humanities\\ Institute of Engineering and Technology\\
Lucknow-226021(Uttar Pradesh), India\\
Email: kksiitr.singh@gmail.com
\end{center}
\begin{abstract}
In this paper, we give direct theorems on  point wise and global approximation by new variants of Bernstein-Durrmeyer operator, introduced by A.-M. et al.\cite{ana2019}.

\end{abstract}
\section{Introduction}
Let $C[0,1]$ be the space of continuous functions defined on $[0,1],$ and let $n\in \mathbb{N}.$ Then, the Bernstein operators $B_n:C[0,1]\longrightarrow C[0,1]$ are defined by
\[B_n(f;x):=\sum_{k=0}^{n}p_{n,k}(x)f\left(\frac{k}{n}\right),\quad 0\leq x\leq 1,\]
where the basis functions $p_{n,k}(x)={n \choose k}x^k(1-x)^{n-k}$ for $k=\overline{0,n}.$
  The Bernstein operators exhibit interesting properties e.g. preservation of convexity, Lipschitz continuity, monotonicity. Although $B_n(f;x)$ are easy to work with, they are not suitable to approximate integrable function as such.
In order to approximate  bounded and integrable functions on
$[0,1],$ Durrmeyer \cite{jld1967} and Lupa\c{s} \cite{AL1972} independently defined the operator
\begin{equation}\label{definitionoperator}D_n(f;x):=\int_{0}^{1}(n+1)\sum_{k=0}^{n}p_{n,k}(x)p_{n,k}(t)f(t)\,\textrm{d}t.
\end{equation}
These operators are obtained by a  suitable modification of Bernstein operators, wherein  the values $f(k/n)$ are replaced by the average values \[(n+1)\int_0^1 p_{n,k}(t)f(t)\,\textrm{d}t\] of $f$ in the interval $[0,1].$
It is worth pointing out that some researchers denote the operators (\ref{definitionoperator}) by $M_n(f;x).$ The operators $D_n(f;x)$ have been extensively studied by
Derriennic \cite{MMD} and other researchers (see \cite{ZDKI1989},\cite{jld1967},\cite{HHGXLZ1991}),\cite{vg2009}, \cite{vghmsri2018}, \cite{srivastavhmvgupta2003}. It turns out that the sequence $D_n(f;x)$ converges point-wise to $f(x)$ at each point $x$ where $f$ is continuous. In fact for continuous function $f$ in $[0,1]$ and $n\geqslant 3$ there holds the inequality(\cite{MMD})
\begin{equation*}\label{Derrienic1}
\sup_{x\in[0,1]}|D_nf(x)-f(x)|\leq 2 \omega_f(1/\sqrt{n}).
\end{equation*}
Moreover, the convergence of $D_n(f;x)$ to $f(x)$ is uniform. Similar to Bernstein operators it turns
out that the order of approximation by the operators $D_n(f;x)$ is at
best, $O(n^{-1})$ however smooth the function may be.

\par
It is well known that (see\cite{vgrpa}) the common linear positive operators, namely, Bernstein operator, Baskakov operator, integral variants of Bernstein operators and others are saturated by order $n^{-1}.$ Consequently, it is not possible to improve the rate of approximation by these operators except for linear functions. Recently a new method  to improve the rate of approximation by decomposition of the weight function $p_{n,k}(x)$  was recently introduced by Khosravian-Arab et al.\cite{Khosravianetal}. The authors of \cite{Khosravianetal} applied this technique to the Bernstein polynomials and were successful to improve the degree of approximation from linear order $O(n^{-1})$ to the  quadratic order $O(n^{-2})$ as well as cubic order $O(n^{-3})$.
Subsequently, this technique was applied to the Bernstein-Durrmeyer operators $D_n(f;x)$ by Ana Maria et al.\cite{ana2019}. Very recently,  Kajla and Acar \cite{KajlaandAcar} has applied this method to $\alpha-$ Bernstein operators. In this paper, we obtain new definite results for the operators studied in \cite{ana2019}.
\par Throughout this paper,  $L_B[0,1]$  will denote the class of Lebesgue integrable and bounded functions in $[0,1].$ By $f(x)\sim g(x)$ we mean there exist absolute constants $A$ and $B$ independent of $x$ such that $f(x)\leqslant B g(x)$ and $g(x)\leqslant A f(x)$ as $n \rightarrow \infty.$ We will adopt the convention $p_{n,k}(x)=0$ for $n<k$ or $k<0.$
\section{First Order  Durrmeyer Operator $D^{M,1}_n(f;x)$}
For $f\in L_B[0,1],$ A.-M. et al. \cite{ana2019} introduced Durrmeyer operator of order I as follows:
\begin{equation}\label{defoperator1}
D^{M,1}_n(f;x):=(n+1)\sum_{k=0}^{n}p_{n,k}^{M,1}(x)\int_{0}^{1}p_{n,k}(u)f(u)\,\textrm{d}u,
\end{equation}
where
\begin{equation*}
p_{n,k}^{M,1}(x)=a(x,n)p_{n-1,k}(x)+a(1-x,n)p_{n-1,k-1}(x),
\end{equation*}
the function
\begin{equation*}
a(x,n)=a_1(n)x+a_0(n)
\end{equation*}
and the unknown sequences $a_0(n),$ $a_1(n)$ are to be determined suitably by assuming the identities $D^{M,1}_n(e_i;x)=e_i(x)$ for $i=0,1$ and for $i=2$ in some cases. We write $D^{M,1}_n(f;x)=A_n(f;x)+B_n(f;x),$ where
\[A_n(f;x)=(n+1)\sum_{k=0}^{n}a(x,n)p_{n-1,k}(x)\int_{0}^{1}p_{n,k}(u)f(u)\,\textrm{d}u,\] and
\[B_n(f;x)=(n+1)\sum_{k=0}^{n}a(1-x,n)p_{n-1,k-1}(x)\int_{0}^{1}p_{n,k}(u)f(u)\,\textrm{d}u.\]
In \cite{ana2019}  A.-M. et al. proved uniform convergence of the sequence $D^{M,1}_n(f;x)$ to $f(x)$ under the cases $2a_0(n) + a_1(n) = 1,$ $a_0(n)\geq 0 $ and  $a_0(n) + a_1(n)\geq 0.$ Further, $D^{M,1}_n(f;x)$ again converges uniformly to $f(x)$ for the case $2a_0(n) + a_1(n) = 1,$ $a_1(n)$ convergent and
$a_0(n) < 0$ or $a_1(n) + a_0(n) < 0.$ The authors of \cite{ana2019} has given following error estimate:
\begin{thm} If $f$ is a bounded function for $x \in [0, 1],$ $a_1(n)$  a bounded
sequence and $2a_0(n) + a_1(n) = 1,$ then
\[\left\|D^{M,1}_n f-f\right\|\leq (3|a_0(n)+1|)(1+\sqrt{2})\omega\left(f; \frac{1}{\sqrt{n}}\right),n\geq 1,\]
where $\|\cdot\|$ is the uniform norm on the interval $[0, 1]$ and $\omega(f,\delta)$ is the first
order modulus of continuity.
\end{thm}
We extend  estimate above by relaxing the conditions on the sequences $a_0(n)$ and $a_1(n).$ Moreover, we obtain a local error estimate giving point wise convergence in following theorem.
\begin{thm}[Rate of Approximation]\label{rateofapproximation0}
Let $D^{M,1}_n(f;x)$ be the operator defined by (\ref{defoperator1}) and $f\in C[0,1].$ If $a_0(n)$ and $a_1(n)$ be two bounded sequences
 such that for each $n,$ $2a_0(n)+a_1(n)=1.$ Then,
\begin{align*}
\big|D^{M,1}_n(f;x)-f(x)\big|&\leq 4(\sqrt{2}+1)\Big(|a_0(n)+a_1(n)|\Big)\omega\left(f,\sqrt{\delta_n(x)}\right),
\end{align*}
where
\begin{equation*}\delta_n(x)=
\left\{
\begin{array}{ll}
\displaystyle{\frac{(3-9x+7x^2)+n\varphi^2(x)}{(n+2)(n+3)}}, & \hbox{$0\leq x \leq 1/2$} \vspace{2mm}\\
\displaystyle{\frac{(1-5x+7x^2)+n\varphi^2(x)}{(n+2)(n+3)}},& \hbox{$1/2\leq x \leq 1.$}
\end{array}
\right.\end{equation*}
\end{thm}
\begin{proof}
Since $D^{M,1}_n(e_0;x)=2a_0(n)+a_1(n),$ so
\begin{align*}
    &D^{M,1}_n(f;x)-(2a_0(n)+a_1(n))f(x)\\
    &=\big(a_0(n)+a_1(n) x\big)A_n(f;x)+\big(a_0(n)+a_1(n) (1-x)\big)B_n(f;x)\\
    &=\big(a_0(n)+a_1(n) x\big)\Big(A_n(f;x)-f(x)A_n(e_0;x)\Big)\\
    &+\big(a_0(n)+a_1(n) (1-x)\big)\Big(B_n(f;x)-f(x)B_n(e_0;x)\Big)\\
    &=\big(a_0(n)+a_1(n) x\big)E_1+\big(a_0(n)+a_1(n) (1-x)\big)E_2,\quad \textrm{say}.
\end{align*}
Now, by the property
\[\omega(f,\lambda\delta)\leq (\lambda+1)\omega(f,\delta), \lambda\geq 0\]
 of $\omega(f,\delta)$ and an application of H\"{o}lder's inequalities for integration and then summation yields
\begin{align}\label{estimateforAn}
    |E_1|&=(n+1)\left|\sum_{k=0}^{n}p_{n-1,k}(x)\int_{0}^{1}p_{n,k}(t)\left(f(t)-f(x)\right)\,\textrm{d}t\right|\nonumber\\
    &\leq (n+1)\sum_{k=0}^{n}p_{n-1,k}(x)\int_{0}^{1}p_{n,k}(t)\omega(f,|t-x|)\,\textrm{d}t\nonumber\\
    &\leq \omega\Big(f,\frac{1}{\delta_n(x)}\Big) (n+1)\sum_{k=0}^{n}p_{n-1,k}(x)\int_{0}^{1}p_{n,k}(t)\big(\delta_n(x)|t-x|+1\big)\,\textrm{d}t\nonumber\\
    &\leq \omega\Big(f,\frac{1}{\delta_n(x)}\Big) (n+1)\sum_{k=0}^{n}p_{n-1,k}(x)\Bigg[\delta_n(x)\Big(\int_{0}^{1}p_{n,k}(t)(t-x)^2\,\textrm{d}t\Big)^{1/2}\nonumber\\
    &\Big(\int_{0}^{1}p_{n,k}(t)\,\textrm{d}t\Big)^{1/2}+\int_{0}^{1}p_{n,k}(t)\,\textrm{d}t\Bigg]\nonumber\\
    &= \omega\Big(f,\frac{1}{\delta_n(x)}\Big)\Bigg[\delta_n(x)\Big((n+1)\sum_{k=0}^{n}p_{n-1,k}(x)\int_{0}^{1}p_{n,k}(t)(t-x)^2\,\textrm{d}t\Big)^{1/2}\nonumber\\
    &+(n+1)\sum_{k=0}^{n}p_{n-1,k}(x)\int_{0}^{1}p_{n,k}(t)(t-x)^2\,\textrm{d}t\Bigg]\nonumber\\
    &=\omega\Big(f,\frac{1}{\delta_n(x)}\Big)\Bigg[2\delta_n(x) \left(\frac{(1-5x+7x^2)+n\varphi^2(x)}{(n+2)(n+3)}\right)+1\Bigg].
\end{align}
Proceeding likewise for $E_2,$ we get
\begin{equation}\label{estimateforBn}|E_2|\leq \omega\Big(f,\frac{1}{\delta_n(x)}\Big)\Bigg[2\delta_n(x) \left(\frac{(3-9x+7x^2)+n\varphi^2(x)}{(n+2)(n+3)}\right)+1\Bigg].
\end{equation}
Finally, we have
\begin{align*}
&\left|D^{M,1}_n(f;x)-f(x)\right|\\
&\leq \big|a_0+a_1 \big|\omega\Big(f,\frac{1}{\delta_n(x)}\Big)\Bigg[2\delta_n(x) \left(\frac{(1-5x+7x^2)+n\varphi^2(x)}{(n+2)(n+3)}\right)+1\Bigg]\\
&+\big|a_0+a_1 (1-x)\big|\omega\Big(f,\frac{1}{\delta_n(x)}\Big)\Bigg[2\delta_n(x) \left(\frac{(3-9x+7x^2)+n\varphi^2(x)}{(n+2)(n+3)}\right)+1\Bigg].
\end{align*}
The proof is now completed.
\end{proof}
\begin{remark}\label{Remark1}
By Theorem\cite{HGonska}, it follows that for $f\in C^1[0,1],$ there holds the error estimate
\begin{align*}
 \left|D^{M,1}_n(f;x)-f(x)\right|&\leq 2\left(\frac{|2-3x|}{n+2}\left|a_0(n)+a_1(n)\right| \|f'\|+\frac{3}{2}\sqrt{\delta_n(x)} \omega \left(f',\sqrt{\delta_n(x)}\right)\right),
\end{align*}
where $\delta_n(x)$ is given in previous theorem. Consequently, for bounded sequences $a_0(n)$ and $a_1(n)$ we find the order
\begin{equation*}
    \left\|D^{M,1}_n(f,\cdot)-f(\cdot)\right\|\leq C \frac{1}{n+2}\left(\|f'\|+\frac{1}{\sqrt{n+2}}\omega \left(f',\sqrt{\frac{1}{n+2}}\right)\right).
\end{equation*}
Hence,
\[\left\|D^{M,1}_n(f,\cdot)-f(\cdot)\right\|\sim\left(\frac{1}{n+2}\right).\]
\end{remark}
We further extend the error estimates by the B\'{e}zier Variant of the operator $D^{M,1}_n(f;x)$ for a larger class, namely the class of functions of bounded variation.
For $\mu\ge 1$ and  $f\in L_B[0,1],$  we define B\'{e}zier variant of the operator $D^{M,1}_n(f;x)$ as
\begin{equation}\label{defoperator1}
D^{M,1}_{n,\mu}(f;x):=(n+1)\sum_{k=0}^{n}Q^{(\mu)}_{n,k}(x)\int_{0}^{1}p_{n,k}(u)f(u)\,\textrm{d}u,
\end{equation}
where
\begin{equation*}
Q^{(\mu)}_{n,k}(x)=(J_{n,k}(x))^{\mu}-(J_{n,k+1}(x))^{\mu}, J_{n,k}(x)=\sum_{j=k}^{n}K_{n,j}(x).
\end{equation*}
For $\mu=1,$ this family of B\'{e}zier operators yield the operators $D^{M,1}_{n,\mu}(f;x)$ studied by Acu et al. \cite{ana2019}.
Let $$R_{n,\mu}(x,u)=(n+1)\sum_{k=0}^nQ^{(\mu)}_{n,k}(x)p_{n,k}(u).$$
Then $D^{M,1}_{n,\mu}(f;x):=\int_{0}^{1}R_{n,\mu}(x,u)f(u)\textrm{d}u.$
\begin{lemma}\cite{ana2019}\label{momentlemma}
We have
$$D^{M,1}_{n}((u-x);x)=\frac{(1-2x)(3a_0(n)+2a_1(n))x}{n+2},$$
\begin{align*}D^{M,1}_{n}((u-x)^2;x)&=\frac{2(3a_1(n)+4a_0(n)-11a_1(n)x+14x^2a_0(n)+11a_1(n)x^2-14a_0(n)x))}{(n+2)(n+3)}\\
&+\frac{2(1-x)x(2a_0(n)+a_1(n)n}{(n+2)(n+3)}.
\end{align*}
\end{lemma}
\begin{remark}
By an application of  Lemma \ref{momentlemma}, we have
\begin{equation}\label{bdd2order}D^{M,1}_{n}((u-x)^2;x)\le \frac{B_{a_0,a_1}}{n+2}\varphi^2(x), x\in [0,1],
\end{equation} where $\varphi^2(x)=x(1-x)$ and $B_{a_0,a_1}$ is positive constant depending on $a_0$ and $a_1.$
\end{remark}
\begin{lemma}
If $f\in C[0,1]$ then,  $\|D^{M,1}_{n}(f)\|\le \|f\|,$ where $\|.\|$ denotes the sup-norm on $[0,1].$
\end{lemma}
\begin{lemma}\label{lemmabddbezier}
If $f\in C[0,1]$ then,  $\|D^{M,1}_{n,\mu}(f)\|\le \mu \|f\|.$
\end{lemma}
\begin{proof}
Using the inequality $|a^\mu-b^\mu|\le \mu |a-b|$ with $0\le a,b \le 1, \mu \ge 1$ and definition of $D^{M,1}_{n,\mu}(f),$ we get for $\mu\ge 1$
\begin{equation}\label{lemma3result}
0<[(J_{n,k}(x))^{\mu}-(J_{n,k+1}(x))^{\mu}]\le \mu [J_{n,k}(x)-J_{n,k+1}(x)]=\mu K_{n,k}(x)
\end{equation} Applying $D^{M,1}_{n,\mu}(f;x)$ and Lemma \ref{lemmabddbezier}, we get
$$\|D^{M,1}_{n,\mu}(f)\|\le \mu \|D^{M,1}_{n}(f)\|\le \mu \|f\|.$$
\end{proof}
Now, we derive a direct result for B\'{e}zier operators in terms of Ditzian Totik modulus of smoothness $\omega_{\varphi}(f,t)$.
\begin{theorem}
 Let  $\varphi(x)=\sqrt{x(1-x)},$ $\mu\ge 1$ and  $x\in[0,1].$ If $f\in C[0,1]$ then
$$|D^{M,1}_{n,\mu}(f;x)-f(x)| \le C \omega_{\varphi}\left(f,\sqrt{\frac{B_{a_0,a_1}}{n+2}}\right),$$
where $C$ is any absolute constant.\end{theorem}
\begin{proof}
We know that $g(t)=g(x)+\int\limits_x^tg'(u)\textrm{d}u.$ So
\begin{equation}\label{eqn11}|D^{M,1}_{n,\mu}(g;x)-g(x)|=\left|D^{M,1}_{n,\mu}\left(\int\limits_x^tg'(u)\textrm{d}u,x\right)\right|.\end{equation}
For any  $x,t\in(0,1),$ we have
\begin{equation}\label{eqn12}\left|\int\limits_x^tg'(u)\textrm{d}u\right|\le \|\varphi g'\|\left|\int\limits_x^t\frac{1}{\varphi(u)}\textrm{d}u\right|.\end{equation}
Therefore,
\begin{eqnarray} \label{eqn13}\left|\int\limits_x^t\frac{1}{\varphi(u)}\textrm{d}u\right|&=&\left|\int\limits_x^t\frac{1}{\sqrt{u(1-u)}}\textrm{d}u\right|\nonumber\\
&\le &\left|\int\limits_x^t\left(\frac{1}{\sqrt{u}}+\frac{1}{\sqrt{1-u}}\right)\textrm{d}u\right|\nonumber\\
&\le & 2\left(|\sqrt{t}-\sqrt{x}|+|\sqrt{1-t}-\sqrt{1-x}|\right)\nonumber\\
&=&2|t-x|\left(\frac{1}{\sqrt{t}+\sqrt{x}}+\frac{1}{\sqrt{1-t}+\sqrt{1-x}}\right)\nonumber\\
&<&2|t-x|\left(\frac{1}{\sqrt{x}}+\frac{1}{\sqrt{1-x}}\right)\nonumber\\
&<&\frac{2\sqrt{2}|t-x|}{\varphi(x)}.
\end{eqnarray}
Combining (\ref{eqn11})-(\ref{eqn13}) and using Cauchy-Schwarz inequality, we have that
\begin{eqnarray*}\label{eqn14}|D^{M,1}_{n,\mu}(g;x)-g(x)| &\le& 2\sqrt{2}\|\varphi g'\|\varphi^{-1}(x)D^{M,1}_{n,\mu}(|t-x|;x)\\
&\le& 2 \sqrt{2}\|\varphi g'\|\varphi^{-1}(x)\left(D^{M,1}_{n,\mu}((t-x)^2;x)\right)^{1/2}.
\end{eqnarray*}
Now using inequality (\ref{bdd2order}), we obtain
\begin{equation}\label{eqn14}
|D^{M,1}_{n,\mu}(g;x)-g(x)|<C\sqrt{\frac{B_{a_0,a_1}}{(n+2)}}\|\varphi g'\|.\end{equation}
Using Lemma \ref{momentlemma} and inequality (\ref{eqn14}), we obtain
\begin{eqnarray}\label{eqn15}|D^{M,1}_{n,\mu}(f;x)-f| &\le& |D^{M,1}_{n,\mu}(f-g;x)|+|f-g|+|D^{M,1}_{n,\mu}(g;x)-g(x)|\nonumber\\
&\le  & C \left(\|f-g\|+\sqrt{\frac{B_{a_0,a_1}}{(n+2)}}\|\varphi g'\|\right).
\end{eqnarray}
Taking infimum on both sides of (\ref{eqn15}) over all $g \in W^2_{\phi}$, we reach to
\begin{equation*}\label{eqn16}|D^{M,1}_{n,\mu}(f;x)-f|\le C K_{\varphi}\left(f,\sqrt{\frac{B_{a_0,a_1}}{(n+2)}}\right)
\end{equation*}
Using relation  $K_{\varphi}(f,t) \sim \omega_{\varphi}(f,t),$  we get the required result.
\end{proof}
Ozarslan and Aktuglu \cite{ozarslan} is considered the Lipschitz-type space with two parameters $\alpha_1\ge 0,\alpha_2>0$, which is defined as
$$Lip_M^{(\alpha_1,\alpha_2)}(\zeta)=\left\{f\in C[0,1]:|f(t)-f(x)|\le C_0\frac{|t-x|^\zeta}{(t+\alpha_1x^2+\alpha_2x)^{\zeta/2}}:t\in[0,1], x\in (0,1)\right\},$$
where $0<\zeta\le 1,$ and $C_0$ is any absolute constant.
\begin{theorem}
Let $f\in Lip_M^{(\alpha_1,\alpha_2)}(\zeta).$ Then for all $x\in (0,1]$, there holds:
\begin{equation*}\label{eqn17}|D^{M,1}_{n,\mu}(f;x)-f|\le C \left(\frac{\mu B_{a_0,a_1}\varphi^2(x)}{(n+2)(\alpha_1x^2+\alpha_2x)}\right)^{\frac{\zeta}{2}},
\end{equation*} where $C$ is any absolute constant.
\end{theorem}
\begin{proof}
Using Lemma \ref{momentlemma} and (\ref{lemma3result}) and H\"{o}lder's inequality with $p=\frac{2}{\zeta}$ and $p=\frac{2}{2-\zeta},$ we get
\begin{align*}
 &|D^{M,1}_{n,\mu}(f;x)-f(x)|\\
 &\le (n+1)\sum_{k=0}^{n}Q^{(\mu)}_{n,k}(x)\int_{0}^{1}p_{n,k}(u)|f(u)-f(x)|\,\textrm{d}u\\
 &\le (n+1)\sum_{k=0}^{n}Q^{(\mu)}_{n,k}(x)\left(\int_{0}^{1}p_{n,k}(u)|f(u)-f(x)|^{\frac{2}{\zeta}}\,\textrm{d}u\right)^{\frac{\zeta}{2}}\\
 &\le  \left[(n+1)\sum_{k=0}^{n}Q^{(\mu)}_{n,k}(x)\int_{0}^{1}p_{n,k}(u)|f(u)-f(x)|^{\frac{2}{\zeta}}\,\textrm{d}u\right]^{\frac{\zeta}{2}}\times\\
&\left[(n+1)\sum_{k=0}^{n}Q^{(\mu)}_{n,k}(x)\int_{0}^{1}p_{n,k}(u)\,\textrm{d}u\right]^{\frac{2-\zeta}{2}}\\
&=\left[(n+1)\sum_{k=0}^{n}Q^{(\mu)}_{n,k}(x)\int_{0}^{1}p_{n,k}(u)|f(u)-f(x)|^{\frac{2}{\zeta}}\,\textrm{d}u\right]^{\frac{\zeta}{2}}\\
&\le C\left((n+1)\sum_{k=0}^{n}Q^{(\mu)}_{n,k}(x)\int_{0}^{1}p_{n,k}(u)\frac{(u-x)^2}{(u+\alpha_1x^2+\alpha_2x)}\textrm{d}u\right)^{\frac{\zeta}{2}}\\
&\le \frac{C}{(\alpha_1x^2+\alpha_2x)^{\frac{\zeta}{2}}}  \left((n+1)\sum_{k=0}^{n}Q^{(\mu)}_{n,k}(x)\int_{0}^{1}p_{n,k}(u)(u-x)^2\textrm{d}u\right)^{\frac{\zeta}{2}}\\
&\le \frac{C}{(\alpha_1x^2+\alpha_2x)^{\frac{\zeta}{2}}} [D^{M,1}_{n,\mu}((u-x)^2;x)]^{\frac{\zeta}{2}}\\
&\le C \left(\frac{\mu B_{a_0,a_1}\varphi^(x)}{(n+2)(\alpha_1x^2+\alpha_2x)}\right)^{\frac{\zeta}{2}}.
\end{align*}
Hence, the desired result follows.
\end{proof}
\section{Rate of convergence}
If we define $$\kappa_{n,\mu}(x,y)=\int\limits_0^yR_{n,k}(x,t)\textrm{d}t.$$ Then, it is obvious that $\kappa_{n,\mu}(x,1)=1$.
\begin{lemma}\label{lem3}
Let $x\in (0,1)$ and $C>2$ then for sufficiently large $n$ we have
$$\kappa_{n,\mu}(x,y)=\int\limits_0^yR_{n,\mu}(x,t)dt\le \frac{C\mu x(1-x)}{n(x-y)^2},0<y<x$$
$$1-\kappa_{n,\mu}(x,z)=\int\limits_z^1R_{n,\mu}(x,t)dt\le \frac{C\mu x(1-x)}{n(z-x)^2},x<z<1.$$
\end{lemma}
\begin{remark} We have
\begin{align*}\kappa_{n,\mu}(x,y)&=\int\limits_0^yR_{n,\mu}(x,t)dt \le \int\limits_0^yR_{n,\mu}(x,t)\frac{(t-x)^2}{(y-x)^2}dt\\&=\frac{D^{M,1}_{n,\mu}((t-x)^2;x)}{(y-x)^2}\le
 \frac{\mu D^{M,1}_n((t-x)^2;x)}{(y-x)^2}\le\frac{\mu B_{a_0,a_1}}{n+2}.\frac{\varphi^2(x)}{(y-x)^2}.\end{align*}
\end{remark}
 Now, we discuss the approximation properties of functions having derivative of bounded variation on $[0,1]$. Let $\mathbb{B}[0,1]$ denote the class of differentiable functions $g$ defined on $[0,1]$, whose derivative $g'$ is of bounded variation on $[0,1]$. The functions $g\in \mathbb{B}[0,1]$ is expressed as $g(x)=\int_0^x h(t)dt+g(0)$, where $h\in \mathbb{B}[0,1],$ i.e., $h$ is a function of bounded variation on $[0,1].$

\begin{theorem}
Let $f\in\mathbb{B}[0,1]$ then for $\mu\ge 1, 0<x<1 $ a and sufficiently large $n$ we have
\begin{align*}\left|D^{M,1}_{n,\mu}(f;x)-f(x)\right|&\le \left(\frac{1}{\mu+1}|f'(x+)+\mu f'(x-)|+|f'(x+)-f'(x-)|\right)\frac{\mu B_{a_0,a_1}}{n+2}\varphi^2(x)\\
&+\frac{\mu B_{a_0,a_1}}{n+2}.\frac{\varphi^2(x)}{x^2}\sum_{k=1}^{\sqrt{n}}\left(\bigvee_{x-\frac{x}{k}}^x(f')_x\right)
+\frac{x}{\sqrt{n}}\left(\bigvee_{x}^{x+\frac{1-x}{\sqrt{n}}}(f')_x\right) \left(\bigvee_{x-\frac{x}{k}}^x(f')_x\right)\\
&+\frac{\mu B_{a_0,a_1}}{n+2}.\frac{\varphi^2(x)}{1-x}\sum_{k=1}^{\sqrt{n}}\left(\bigvee_x^{x+\frac{1-x}{k}}(f')_x\right)
+\frac{1-x}{\sqrt{n}}\left(\bigvee_{x}^{x+\frac{1-x}{\sqrt{n}}}(f')_x\right).
    \end{align*}
\[(f')_x(t)=\left\{
\begin{array}{lll}
f'(t)-f'(x-),  & \hbox{$0\le t <x$} \\
0,& \hbox{$t=x$}\\
f'(t)-f'(x+),& \hbox{$x<t\le1$.}
\end{array}\right.\]
\end{theorem}
\begin{proof}
  As we know that $D^{M,1}_{n,\mu}(1;x)=1.$ Therefore, we have
\begin{eqnarray}\label{eqn0000}
D^{M,1}_{n,\mu}(f;x)-f(x)&=&\int\limits_0^1R_{n,\mu}(x,u)(f(u)-f(x))\,\textrm{d}u\nonumber\\
&=&\int\limits_0^1R_{n,\mu}(x,u)\int\limits_x^uf'(t)\,dt\,\textrm{d}u
\end{eqnarray}
Since $f\in \mathbb{B}[0,1]$, we may write
\begin{eqnarray}\label{eqn1111}
f'(t)&=&\frac{f'(x+)+\mu f'(x-)}{\mu+1}+(f')_x(t)+\frac{f'(x+)- f'(x-)}{2}\left(\mbox{sign}(t-x)+\frac{\mu-1}{\mu+1}\right)\nonumber\\&+&\delta_x(t)
\left(f'(t)-\frac{f'(x+)+\mu f'(x-)}{2}\right)
\end{eqnarray} where
\[\mbox{sign}(t)=\left\{
\begin{array}{lll}
1,  & \hbox{$t>0$} \\
0,& \hbox{$t=0$}\\
-1,& \hbox{$t<0$.}
\end{array}\right.\]
and
\[\delta_x(t)=\left\{
\begin{array}{ll}
1,  & \hbox{$t=x$} \\
0,& \hbox{$t\neq x$.}
\end{array}\right.\]
Putting  the value of $f'(t)$ form (\ref{eqn1111}) in \ref{eqn0000}, we get estimates corresponding to four terms of (\ref{eqn1111}), say $I_1, I_2, I_3$ and $I_4$ respectively. Obviously,
$$I_4=\int\limits_0^1\left(\int\limits_x^u  \left(f'(t)-\frac{f'(x+)+\mu f'(x-)}{2}\right) \delta_x(t)dt\right) R_{n,\mu}(x,u)\textrm{d}u=0.$$
Now,
Using Cauchy's Schwarz inequality and Lemma \ref{lem3}, we obtain
\begin{align*}
I_1=\int\limits_0^1\left(\int\limits_x^u \frac{f'(x+)+\mu f'(x-)}{\mu+1}dt\right)R_{n,\mu}(x,u)\textrm{d}u&=\frac{f'(x+)+\mu f'(x-)}{\mu+1}\int\limits_0^1(u-x)R_{n,\mu}(x,u)\textrm{d}u\\
&=\frac{f'(x+)+\mu f'(x-)}{\mu+1}D^{M,1}_{n,\mu}((u-x);x)\\
&\le \frac{B_{a_0,a_1}}{\sqrt{n+2}} \frac{f'(x+)+\mu f'(x-)}{\mu+1}\varphi(x).
\end{align*}
Next, again using Cauchy's Schwarz inequality and Lemma \ref{lem3}, we obtain
\begin{align*}I_3&=\int\limits_0^1\left(\int\limits_x^u \left(\frac{f'(x+)- f'(x-)}{2}\right)\left(\mbox{sign}(t-x)+\frac{\mu-1}{\mu+1}\right)dt\right)
R_{n,\mu}(x,u)\textrm{d}u\\
&=\left(\frac{f'(x+)- f'(x-)}{2}\right)\Bigg[-\int\limits_0^x\left(\int\limits_u^x \left(\mbox{sign}(t-x)+\frac{\mu-1}{\mu+1}\right)dt\right)R_{n,\mu}(x,u)\textrm{d}u\\
&+\int\limits_x^1\left(\int\limits_x^u \left(\mbox{sign}_{\mu}(t-x)+\frac{\mu-1}{\mu+1}\right)dt\right)
R_{n,\mu}(x,u)\textrm{d}u\Bigg]\\
&\le \left|f'(x+)- f'(x-)\right|\int\limits_0^1|u-x|R_{n,\mu}(x,u)\textrm{d}u\\
&\le \left|f'(x+)- f'(x-)\right|D^{M,1}_{n,\mu}(|u-x|;x)\\
&\le \frac{B_{a_0,a_1}}{\sqrt{n+2}} \left|f'(x+)- f'(x-)\right|\varphi(x).
\end{align*}
Now, we estimate $I_2$ as follows:
\begin{align*}
I_2=\int\limits_0^1\left(\int\limits_x^u (f')_x(t)dt\right)R_{n,\mu}(x,u)\textrm{d}u&=\int\limits_0^x\left(\int\limits_x^u (f')_x(t)dt\right) R_{n,\mu}(x,u)\textrm{d}u\\
&+\int\limits_x^1\left(\int\limits_x^u (f')_x(t)dt\right)
R_{n,\mu}(x,u)\textrm{d}u\\
&=I_5+I_6,\mbox{say}.
\end{align*}
Using Lemma \ref{lem3} and definition of $\kappa_{n,\mu}(x,u),$ we may write
\begin{align*}I_5&=\int\limits_0^x\left(\int\limits_x^u (f')_x(t)dt\right) \frac{d}{\textrm{d}u}\kappa_{n,\mu}(x,u)\textrm{d}u\\
\end{align*}
  Integrating by parts, we obtain
  \begin{align*}|I_5|&\le \int\limits_0^x|(f')_x(u)|\kappa_{n,\mu}(x,u)\textrm{d}u\\
&\le \int\limits_0^{x-x/\sqrt{n}}|(f')_x(u)|\kappa_{n,\mu}(x,u)\textrm{d}u+\int\limits_{x-x/\sqrt{n}}^x|(f')_x(u)|\kappa_{n,\mu}(x,u)\textrm{d}u\\
&=I_7+I_8,\mbox{say}.
\end{align*}
  In view of facts $(f')_x(x)=0$ and $\kappa_{n,\mu}(x,u)\le 1,$  we get
  \begin{align*}I_8&= \int\limits_0^x|(f')_x(u)-(f')_x(x)|\kappa_{n,\mu}(x,u)\textrm{d}u\\
&\le \int\limits_{x-x/\sqrt{n}}^x\left(\bigvee_{u}^x(f')_x\right)\textrm{d}u\\
&\le \left(\bigvee_{u}^x(f')_x\right) \int\limits_{x-x/\sqrt{n}}^x \textrm{d}u=\frac{x}{\sqrt{n}}\left(\bigvee_{u}^x(f')_x\right).
\end{align*}
Using Lemma \ref{lem3}, definition of $\kappa_{n,\mu}(x,u),$ and transformation $u=x-\frac{x}{t}$ we may write
   \begin{align*}I_7&\le \frac{\mu B_{a_0,a_1}}{n+2}\varphi^2(x) \int\limits_0^{x-x/\sqrt{n}}|(f')_x(u)-(f')_x(x)|\frac{\textrm{d}u}{(u-x)^2}\\
  & \le \frac{\mu B_{a_0,a_1}}{n+2}\varphi^2(x) \int\limits_0^{x-x/\sqrt{n}}\left(\bigvee_{u}^x(f')_x\right)\frac{\textrm{d}u}{(u-x)^2}\\
    &\le \frac{\mu B_{a_0,a_1}}{n+2}\frac{\varphi^2(x) }{x^2} \int\limits_1^{\sqrt{n}} \left(\bigvee_{x-\frac{x}{t}}^x(f')_x\right)dt\\
   &\le \frac{\mu B_{a_0,a_1}}{n+2}\frac{\varphi^2(x) }{x^2}\sum_{k=1}^{|\sqrt{n}|}\left(\bigvee_{x-\frac{x}{t}}^x(f')_x\right).
\end{align*}
 Combining the estimates of $I_7$ and $I_8,$ we have
    \begin{align*}|I_5|&\le \frac{\mu B_{a_0,a_1}}{n+2}\frac{\varphi^2(x) }{x^2}\sum_{k=1}^{|\sqrt{n}|}\left(\bigvee_{x-\frac{x}{t}}^x(f')_x\right)
   + \frac{x}{\sqrt{n}}\left(\bigvee_{u}^x(f')_x\right)
   \end{align*}
   In order to estimate $I_6,$ we use integration by parts, Lemma \ref{lem3} and transformation $z=x+\frac{1-x}{\sqrt{n}}.$ Therefore,we proceed as follows:
   \begin{align*}|I_6|&=\left|\int\limits_x^1\left(\int\limits_x^u (f')_x(t)dt\right)R_{n,\mu}(x,u)\textrm{d}u\right|\\
  &=\Bigg|\int\limits_x^z\left(\int\limits_x^u (f')_x(t)dt\right)\frac{\partial}{\partial u}(1-\kappa_{n,\mu}(x,u))\textrm{d}u+\int\limits_z^1\left(\int\limits_x^u (f')_x(t)dt\right)\frac{\partial}{\partial u}(1-\kappa_{n,\mu}(x,u))\textrm{d}u\Bigg|\\
   &=\Bigg|\left[\int\limits_x^u (f')_x(t)dt(1-\kappa_{n,\mu}(x,u))\right]_{x}^{z}-\int\limits_x^z (f')_x(u)(1-\kappa_{n,\mu}(x,u))\textrm{d}u\\
   &+\left[\int\limits_x^u (f')_x(t)dt(1-\kappa_{n,\mu}(x,u)) \right]_{z}^{1}-\int\limits_z^1 (f')_x(u)(1-\kappa_{n,\mu}(x,u))\textrm{d}u\Bigg|\\
   &=\left|\int\limits_x^z (f')_x(u)(1-\kappa_{n,\mu}(x,u))\textrm{d}u+\int\limits_z^1 (f')_x(u)(1-\kappa_{n,\mu}(x,u))\textrm{d}u\right|\\
   &\le \frac{\mu B_{a_0,a_1}}{n+2}.\varphi^2(x)  \int\limits_z^1\left(\bigvee_{x}^u(f')_x\right)(u-x)^{-2}\,\textrm{d}u+\int\limits_x^z
    \left(\bigvee_{x}^u(f')_x\right)\,\textrm{d}u\\
    &\le \frac{\mu B_{a_0,a_1}}{n+2}.\varphi^2(x) \int\limits_{x+\frac{1-x}{\sqrt{n}}}^1\left(\bigvee_{x}^u(f')_x\right)(u-x)^{-2}\,\textrm{d}u
    +\frac{1-x}{\sqrt{n}}
    \left(\bigvee_{x}^{x+\frac{1-x}{\sqrt{n}}}(f')_x\right).
    \end{align*}
Now, substituting  $t=\frac{1-x}{t-x},$ we get
\begin{align*}|I_6| &\le \frac{\mu B_{a_0,a_1}}{n+2}.\varphi^2(x) \int\limits_{1}^{\sqrt{n}}\left(\bigvee_{x}^{x+\frac{1-x}{t}}(f')_x\right)(1-x)^{-1}\,dt+
\frac{1-x}{\sqrt{n}}\left(\bigvee_{x}^{x+\frac{1-x}{\sqrt{n}}}(f')_x\right)\\
   &\le \frac{\mu B_{a_0,a_1}}{n+2}\frac{\varphi^2(x)}{1-x} \sum_{k=1}^{\sqrt{n}}\left(\bigvee_{x}^{x+\frac{1-x}{k}}(f')_x\right)+\frac{1-x}{\sqrt{n}}\left(\bigvee_{x}^{x+\frac{1-x}{\sqrt{n}}}(f')_x\right).
   \end{align*}
Combining the estimates of $I_1-I_8,$ we get the desired result. Hence the proof follows.
\end{proof}

\section{Second Order  Durrmeyer Operator $D^{M,2}_n(f;x)$}
Assuming a weight function $p_{n,k}^{M,2}(x)$ of the form
\begin{equation*}
p_{n,k}^{M,2}(x)=a_0(x,n)p_{n-2,k}(x)+a_1(x,n)p_{n-2,k-1}(x)+a_2(x,n)p_{n-2,k-2}(x),
\end{equation*}
 yet to be determined, A.-M. et al.\cite{ana2019}
have defined the operator $D^{M,2}_n(f;x)$ by
\[D^{M,2}_n(f;x)=(n+1)\sum_{k=0}^{n}p_{n,k}^{M,2}(x)\int_{0}^{1}p_{n,k}(u)f(u)\textrm{d}u,n\geq 3.\]
By setting $b_0(n)=\frac{3}{2},$ $b_1(n)=-n,$ $b_2(n)=n-2$ and $d_0(n)=2(n-2)$ the authors of \cite{ana2019} obtained the order of approximation $O\left(n^{-2}\right).$ However, we prove the rate of approximation by obtaining it in terms of modulus of continuity.
For the sake of convenience we rewrite $D^{M,2}_n(f;x)$ as
\begin{equation}\label{operator2}
D^{M,2}_n(f;x)=\sum_{j=0}^{2}a_j(x,n)A_j(f;x),
\end{equation}
$a_0(x,n)=a_n+b_nx+c_nx^2,$  $a_1(x,n)=d_n x(1-x)$ and $a_2(x,n)=a_0(1-x,n).$ The sequences $a_n,b_n,d_n$
with
\begin{equation}\label{operator2,1}A_j(f;x)=(n+1)\sum_{k=0}^{n}p_{n-2,k-j}(x)\int_{0}^{1}p_{n,k}(u)f(u)\textrm{d}u.\end{equation}
By direct calculations, we get $A_j(e_0;x)=1$
for each $j=0,1,2.$ Thus, we have $a_0(x,n)+a_1(x,n)+a_2(x,n)=1$ which in turn implies
$2a_n+b_n+c_n+(d_n-2c_n)\varphi^2(x)=1.$ The requirement $D^{M,2}_n(e_0;x)=1$ leads to the limits
\begin{equation}\label{lim1}2a_n+b_n+c_n \longrightarrow 1\end{equation}
 and
\begin{equation}\label{lim2}d_n-2c_n \longrightarrow 0\end{equation}
 as $n \longrightarrow \infty.$
Next, we have again by straight forward calculations
\begin{equation}\label{e1forcomponents}A_j(e_1;x)=\frac{(n-2) x+j+1}{n+2}.
\end{equation}
So that using the definition (\ref{operator2}) we obtain
\begin{align*}
D^{M,2}_n(e_1;x)&=\frac{4a_n+3(b_n+c_n)+(-2b_n-6c_n+2d_n+(2a_n+b_n+c_n)(n-2))x}{(n+2)}\\
&+\frac{(4c_n-2d_n+(n-2)(d_n-2c_n))x^2+(2c_n-d_n)x^3}{(n+2)}.
\end{align*}
Therefore, if we impose the condition $D^{M,2}_n(e_1;x)=1$ and make use of (\ref{lim1})-(\ref{lim2}) we get $a=3/2,$ $b=-2-c$ and $c=n+8=d/2.$
Hence, we have
\begin{align}\label{a0}a_0(x,n)&=\frac{3}{2}-2x-(n+8)\varphi^2(x),\\
\label{a1} a_1(x,n)&=2(n+8) \varphi^2(x),\\
\label{a2} a_2(x,n)&=-\frac{1}{2}+2x-(n+8) \varphi^2(x).
\end{align}
Next, we have that
\begin{align}\label{e2forcomponents}A_0(e_2;x)&=\frac{\left(n^2-5 n+6\right) x^2+4 (n-2) x+2}{(n+2) (n+3)},\\
A_1(e_2;x)&=\frac{\left(n^2-5 n+6\right) x^2+6 (n-2) x+6}{(n+2) (n+3)},\\
A_2(e_2;x)&=\frac{\left(n^2-5 n+6\right) x^2+8 (n-2) x+12}{(n+2) (n+3)}.
\end{align}
Therefore, utilization of $a_i(x,n)$ and $A_i(e_2;x), i=0,1,2$ in (\ref{operator2}), it follows that
\[D^{M,2}_n(e_2;x)=x^2-\frac{3}{(n+2)(n+3)}.\]

\section{Auxiliary Results on the Operator $D^{M,2}_n(f;x)$}

\begin{lemma}\label{orederpolynomial} If $\mu_k(x)=A_0(t^k;x),$  $\lambda_k(x)=A_1\left(t^k;x\right)$ and $\eta_k(x)=A_2(t^k;x)$ then,
\begin{align}\label{recurrence}(n+k+2)D^{M,2}_n(t^{k+1};x)&=\varphi^2(x)\frac{d}{dw}D^{M,2}_n(t^k,w)\Big|_{w=x}+(k+3+(n-2)x)D^{M,2}_n(t^k;x)\nonumber\\
&-\varphi^2(x)\Big(a'_0(x,n)\mu_k(x)+a'_1(x,n)\lambda_k(x)\nonumber\\
&+a'_2(x,n)\eta_k(x)\Big)-2a_0(x,n)\mu_k(x)-a_1(x,n)\lambda_k(x).
\end{align}
\end{lemma}
\begin{proof}
Denote $S(t^k;x):=a_0(x,n)A_0(t^k;x),$ $T(t^k;x):=a_1(x,n)A_1(t^k;x),$ and $U(t^k;x):=a_2(x,n)A_2(t^k;x)$ so that
\begin{equation}
D^{M,2}_n(t^k;x)=(S+T+U)(t^k;x).
\end{equation}
We have that
\[\frac{d}{dx}p_{n-2,r}(x)=\frac{(r-(n-2)x)}{\varphi^2(x)}p_{n-2,r}(x).\]
Therefore,
\begin{align}\label{recurrforA}
\varphi^2(x)\frac{d}{dw}A_0(t^k;w)\Big|_{w=x}&=(n+1)\sum_{r=0}^{n}p_{n-2,r}(x)\int_0^1 (r-(n-2)x)u^k p_{n,r}(u)\textrm{d}u\nonumber\\
&=(n+1)\sum_{r=0}^{n}p_{n-2,r}(x)\int_0^1\varphi^2(u)p'_{n,r}(u)\textrm{d}u\nonumber\\
&+n A_0(t^{k+1};x)-(n-2)x A_0(t^k;x).
\end{align}
Integration by parts, then yields
\begin{align*}
&(n+1)\sum_{r=0}^{n}p_{n-2,r}(x)\int_0^1 \varphi^2(u)p'_{n,r}(u)\textrm{d}u\\
&=(n+1)\sum_{r=0}^{n}p_{n-2,r}(x)\Big[(k+2)\int_0^1u^{k+1} p_{n,r}(u)-(k+1)\int_0^1u^k p_{n,r}(u)\Big]\\
&=(k+2)A_0(t^{k+1};x)-(k+1)A_0(t^k;x).
\end{align*}
Collecting these results and using in (\ref{recurrforA}), the identity
\begin{equation*}\label{An}
\varphi^2(x)A_0(t^k;x)=(n+k+2)A_0(t^{k+1};x)-(k+1+(n-2)x)A_0(t^k;x)
\end{equation*}
is obtained. Finally, writing $\mu_k(x)$ for $A_0(t^k;x)$ we get
\begin{align}\label{Sn}
\varphi^2(x)\frac{d}{dw}S_n(t^k;w)\Big|_{w=x}&=\varphi^2(x)a'_0(x,n)\mu_k(x)+a_0(x,n)\Big((n+k+2)\mu_{k+1}(x)\nonumber\\
&-(k+1+(n-2)x)\mu_k(x)\Big).
\end{align}
Proceeding likewise and denoting $\lambda_k(x)$ for $A_1\left(t^k;x\right)$  we get identity
\begin{equation*}\label{recurrforB}
\varphi^2(x)\frac{d}{dw}A_1(t^k;w)\Big|_{w=x}=(n+k+2)\lambda_{k+1}(x)-(k+2+(n-2)x)\lambda_k(x).
\end{equation*}
Therefore,
\begin{align}\label{Tn}
\varphi^2(x)\frac{d}{dw}T_n(t^k;w)\Big|_{w=x}&=\varphi^2(x)a'_1(x,n)\lambda_k(x)+a_1(x,n)\Big((n+k+2)\lambda_{k+1}(x)\nonumber\\
&-(k+2+(n-2)x)\lambda_k(x)\Big).
\end{align}
And
\begin{align}\label{Un}
\varphi^2(x)\frac{d}{dw}U_n(t^k;w)\Big|_{w=x}&=\varphi^2(x)a'_2(x,n)\eta_k(x)+a_2(x,n)\Big((n+k+2)\eta_{k+1}(x)\nonumber\\
&-(k+3+(n-2)x)\eta_k(x)\Big),
\end{align}
where $\eta_k(x)=A_2(t^k;x).$
Finally, combining the identities (\ref{Sn})-(\ref{Un}), we get the recurrence relation:
\begin{align*}\label{Un}
\varphi^2(x)\frac{d}{dw}D^{M,2}_n(t^k,w)\Big|_{w=x}&=\varphi^2(x)\left(a'_0(x,n)\mu_k(x)+a'_1(x,n)\lambda_k(x)+a'_2(x,n)\eta_k(x)\right)\\
&+2a_0(x,n)\mu_k(x)+a_1(x,n)\lambda_k(x)+(n+k+2)D^{M,2}_n(t^{k+1};x)\\
&-(k+3+(n-2)x)D^{M,2}_n(t^k;x).
\end{align*}
Rearrangement of the terms of above equation, yields (\ref{recurrence}). Hence, the proof is completed.
\end{proof}
\begin{corollary}\label{orderofmoments}
Suppose $D^{M,2}_n(t^k;x)=\sum_{j=0}^{k}\alpha_j^{(k)}x^j,$ then
\[\alpha_k^{(k)}=
\prod_{j=1}^{k}\left(\frac{n-j-1}{n+j+1}\right)=O(1).\] Further, by an induction on $k,$ it follows that
\[D^{M,2}_n(t^k;x)=x^k+O\left(\frac{1}{n^2}\right).\]
\end{corollary}
\begin{corollary}\label{centralmoments}
By linearity of the operator $D^{M,2}_n(f;x)$ and induction on $m,$ it follows immediately that
\[D^{M,2}_n(e_1-x;x)=0,\]
\[D^{M,2}_n((e_1-x)^2;x)=-\frac{3}{(n+2)(n+3)}\] and
\[D^{M,2}_n((e_1-x)^{2m};x)\sim \frac{\varphi^{2m}}{(n+2)(n+3)}.\]
\end{corollary}
\begin{corollary}\label{momentsforAi}
Proceeding in a similar manner, we obtain the recurrence relation
\begin{align*}(n+r+2)A_0((e_1-x)^{r+1};x)&=\varphi^2(x)\frac{d}{dw}A_0((e_1-x)^r,w)\Big|_{w=x}-xA_0((e_1-x)^r;x)\\
&+\left((r+1)(1-2x)+r\varphi^2(x)\right)A_0((e_1-x)^{r-1};x),
\end{align*}
$r\geqslant 1.$
Similar recurrence relations can be established for $A_1(f;x)$ and $A_2(f;x).$
Consequently, by induction for each $m\in \mathbb{N}$
\begin{equation}\label{Aj2m}A_j\left((e_1-x)^{2m},x\right)\sim \frac{\varphi^{2m}(x)}{(n+2)...(n+m+1)}.
\end{equation}
\end{corollary}
\section{Convergence and Order of Approximation}
\begin{thm}\label{Vornovaskayatype}
If $f''(x)$ exists and continuous at $x\in [0,1],$ and conditions (\ref{a0})-(\ref{a2}) hold, then
\begin{equation*}
\lim_{n\rightarrow \infty}(n+2)(n+3)\left(D^{M,2}_n(f;x)-f(x)\right)=-\frac{3}{2}f''(x).
\end{equation*}
\end{thm}
\begin{proof}
We apply the operator $D^{M,2}_n(f;x)$ to the expansion
\[f(t)-f(x)=(t-x)f'(x)+\frac{f''(x)}{2}(t-x)^2+\frac{f''(\xi)-f''(x)}{2}(t-x)^2,\]
where $\xi$ lies between $t$ and $x.$ Making use of the values of $D^{M,2}_n\left((e_1-x),x\right),$ $D^{M,2}_n\left((e_1-x)^2,x\right)$ leads to
\[D^{M,2}_n(f;x)-f(x)=-\frac{3/2}{(n+2)(n+3)}f''(x)+\frac{1}{2}D^{M,2}_n\left(\left(f''(\xi)-f''(x)\right)(t-x)^2,x\right).\]
For $\epsilon>0$ there exists a $\delta>0$ such that $|t-x|<\delta$ implies $|f''(\xi)-f''(x)|< \frac{\epsilon}{n}.$ And for $|t-x|\geq \delta,$ there holds
$|f''(\xi)-f''(x)|<2K,$ where $K=\|f\|.$ Thus,
\begin{align*}
D^{M,2}_n\left(\left(f''(\xi)-f''(x)\right)(t-x)^2,x\right)&=D^{M,2}_n\left(\left(f''(\xi)-f''(x)\right)(t-x)^2\chi_x(t),x\right)\\
&+D^{M,2}_n\left(\left(f''(\xi)-f''(x)\right)
(t-x)^2(1-\chi_x(t)),x\right)\\
&=E_1+E_2,\,\textrm{say}.
\end{align*}
where $\chi_x(t)$ is the characteristic function of the interval $(x-\delta,x+\delta).$
Now,
\begin{align*}
\left|E_1\right|&\leq \sum_{j=0}^{2}\left|a_j(x,n)\right| \left|A_j\left(\left(f''(\xi)-f''(x)\right)(t-x)^2\chi_x(t),x\right)\right|\\
&\leq \sum_{j=0}^{2}C_j \left|a_j(x,n)\right|\epsilon\frac{\varphi^2(x)}{(n+2)(n+3)}\\
&\sim \frac{\epsilon\varphi^2(x)}{(n+2)(n+3)}.
\end{align*}
Similarly, for any $s\geq 4$ we have
\begin{align*}
\left|E_2\right|&\leq 2K\sum_{j=0}^{2}\left|a_j(x,n)\right| \left|A_j\left(\frac{(t-x)^{2s+2}}{\delta^{2s}},x\right)\right|\\
&\leq 2K\sum_{j=0}^{2}C_j \left|a_j(x,n)\right|\frac{\varphi^{2s+2}(x)}{\delta^{2s}(n+2)(n+3)...(n+s+1)}\\
&\sim \frac{\varphi^{2s+2}(x)}{\delta^{2s}(n+2)(n+3)...(n+s)}.
\end{align*}
Choosing $\delta=\frac{1}{n^{\alpha}},$ where $0<\alpha <1/2$ we get
$|E_1|\sim \frac{1}{(n+2)(n+3)}\epsilon $ and $|E_2|\sim \frac{n^{-\alpha}}{(n+2)(n+3)},$ $\alpha >0.$ Therefore,
\begin{align*}
\lim_{n\rightarrow \infty}(n+2)(n+3)\left(D^{M,2}_n(f;x)-f(x)+3f'''(x)\right)&=\epsilon+\lim_{n\rightarrow \infty}n^{-\alpha}.
\end{align*}
The proof now follows in view of the arbitrariness of $\epsilon$ and positivity of $\alpha.$
\end{proof}
\begin{thm}\label{ordinarymodulus} There holds
\begin{equation*}\left|D^{M,2}_n(f;x)-f(x)\right|\leq 2 M_2(x,n)\omega\big(f,\delta_n(x)\big),\end{equation*}
where $M_2(x,n)=\sum_{j=0}^{2}\left|a_j(x,n)\right|,$ and
\begin{equation*}\delta^2_n(x)=
\left\{
\begin{array}{ll}
\displaystyle{\frac{(6-17x+12x^2)+n\varphi^2(x)}{(n+2)(n+3)}}, & \hbox{$0\leq x \leq 1/2$} \vspace{2mm}\\
\displaystyle{\frac{(1-7x+12x^2)+n\varphi^2(x)}{(n+2)(n+3)}},& \hbox{$1/2\leq x \leq 1$.}
\end{array}
\right.\end{equation*}
Moreover, if the sequences $a_j(x,n), j=0,1,2$ are bounded, then
\[\left|D^{M,2}_n(f;x)-f(x)\right|\leq C \omega\left(f,\sqrt{\frac{\varphi^2(x)}{n}}\right).\]
\end{thm}
\begin{proof}
In view of  $a_0(x,n)+a_1(x,n)+a_2(x,n)=1,$  $D^{M,2}_n\left((t-x)^j,x\right)=0,$ $j=0,1$ and $D^{M,2}_n(1;x)=1,$ we can write
\begin{align*}
D^{M,2}_n(f;x)-f(x)&=\sum_{j=0}^{2}a_j(x,n)\left(A_j(f;x)-f(x)\right)=\sum_{j=0}^{2}E_j, \quad \textrm{say}.
\end{align*}
For $E_j$ we have
\begin{align*}
\left|E_j\right|&\leq \left|a_j(x,n)\right|\left(1+\frac{1}{\delta}\sqrt{A_j\left((e_1-x)^2,x\right)}\right)\omega\big(f,\delta_n(x)\big).
\end{align*}
By (\ref{e2forcomponents}) and (\ref{e1forcomponents}), we obtain
\[A_0\left((e_1-x)^2,x\right)=\frac{-2 (n-12) x^2+2 (n-7) x+2}{(n+2) (n+3)},\]
\[A_1\left((e_1-x)^2,x\right)=\frac{-2 (n-12) x^2+2 (n-12) x+6}{(n+2) (n+3)}\] and
\[A_2\left((e_1-x)^2,x\right)=\frac{-2 (n-12) x^2+2 (n-17) x+12}{(n+2) (n+3)}.\]
It is straight forward to verify that for each $j=0,1,2$
\begin{equation*}\left|A_j\left((e_1-x)^2,x\right)\right|\leq
\left\{
\begin{array}{ll}
\displaystyle{\frac{(6-17x+12x^2)+n\varphi^2(x)}{(n+2)(n+3)}}, & \hbox{$0\leq x \leq 1/2$} \vspace{2mm}\\
\displaystyle{\frac{(1-7x+12x^2)+n\varphi^2(x)}{(n+2)(n+3)}},& \hbox{$1/2\leq x \leq 1$.}
\end{array}
\right.\end{equation*}
The proof now follows by choosing $\delta_n(x)=\max_{0\leq j \leq 2}\left|A_j\left((e_1-x)^2,x\right)\right|.$
\end{proof}
Theorem \ref{ordinarymodulus} is of theoretical interest more than that of application. Actually, the second estimate is exactly same as those for Bernstein polynomials.  However, next theorem provides the order of approximation explicitly in terms of the smoothness of the function. We need the  Ditzian Totik modulus of smoothness for $\phi(x)=\sqrt{x(1-x)}$ which is given by
$$ \omega_{\varphi}(f,t)=\sup_{0<h\le t}\sup_{x\pm h\varphi(x)/2\ge 0}\left\{\left|f(x+ h\varphi(x)/2)-f(x-h\varphi(x)/2)\right|\right\}$$
 and appropriate Peter's K functional is defined as
$$K_{\varphi}(f,t)=\inf_{g\in W_{\varphi}^2}\{\|f-g\|+t\|\varphi g'\|+t^2\|g'\|\},$$ where $t>0$, the norm $\|.\|$ is the sup- norm on $[0,1]$ and
 $W_{\varphi}^2=\{g:g\in AC_{loc}, \|\varphi g'\|<\infty, \|g'\|<\infty\}.$
By (\cite{Ditzian1987}, Theorem 3.1.2), it is known that  $K_{\varphi}(f,t) \sim \omega_{\varphi}(f,t)$ and  there exists an absolute constant $C>0$ such that
\begin{equation} \label{rel}C^{-1}\omega_{\varphi}(f,t) \le K_{\varphi}(f,t)\le C \omega_{\varphi}(f,t).\end{equation}
\begin{thm}\label{rate of approximation} Let $f\in C[0,1]$ and $a_j(x,n),j=0,1,2$ be sequences of order at most $n.$ Then there exists a constant $C$ independent of $f,x$ and $n$ such that
\[\|D^{M,2}_n(f)-f\|\leq C \omega^6_{\varphi}\left(f,\frac{1}{n^{1/3}}\right).\]
%\[\|D^{M,2}_n(f)-f\|\leq C \omega_{\varphi}^4\left(f,\frac{1}{\sqrt{n}}\right)+\frac{1}{n^2}\|f\|.\]
\end{thm}
\begin{proof}
We choose a function $g\in C^6[0,1]$ such that
\begin{equation}\label{g function}
\|f-g\|\leq C\omega^6_{\varphi}\left(f,\frac{1}{n^{1/3}}\right).
\end{equation}
Then, an application of $D^{M,2}_n$ to the Taylor's theorem we have that
\begin{align*}
D^{M,2}_n(g;x)-g(x)&=
\sum_{j=2}^{5}\frac{g^{j}(x)}{j!}D^{M,2}_n\left((e_1-x)^j,x\right)+R(g,t;x),\\
\end{align*}
where $R(g,t;x)=D^{M,2}_n\left(\frac{(t-x)^6}{6!}g^{(6)}(\xi),x\right)$ and $\xi$ lies between $t$ and $x.$
By Cor. 2 to lemma\ref{centralmoments},
\begin{align*}
\left|\sum_{j=3}^{5}\frac{g^{j}(x)}{j!}D^{M,2}_n\left((e_1-x)^j,x\right)\right|\leq C \sum_{j=3}^{5}\frac{\left|g^{j}(x)\right|}{j!(n+2)(n+3)}.
\end{align*}
By estimate (\ref{Aj2m}), we get
\begin{align*}
\left|R(g,t;x)\right|&\leq C\frac{\left\|g^{(6)}\right\|}{6!}\sum_{i=0}^{2}\left|a_i(x,n)\right|\frac{\varphi^{6}(x)}{(n+2)(n+3)(n+4)}\\
&\leq M(x,n)\big\|g^{(6)}\big\|\frac{\varphi^{6}(x)}{(n+2)(n+3)(n+4)}\leq C \frac{1}{n^2}\big\|\varphi^{6}g^{(6)}\big\|,
\end{align*}
where $M(x,n)=\sum_{j=0}^{2}\left|a_j(x,n)\right|.$
Collecting these estimates and then using Goldberg and Meir property, we get
\begin{align*}
\|D^{M,2}_n(f,\cdot)-f(\cdot)\|&\leq C\|f-g\|+\frac{1}{n^2}\sum_{j=3}^{5}\frac{\big\|g^{j}\big\|}{j!}+\frac{1}{n^2}\big\|\varphi^{6}g^{(6)}\big\|\\
&\leq C\left(\|f-g\|+\frac{1}{n^2}\big\|\varphi^{6}g^{(6)}\big\|\right)\leq C K_{6,\varphi}\left(f,\frac{1}{n^2}\right).
\end{align*}
Finally, using equivalence of $K_{6,\varphi}\left(f,\frac{1}{n^2}\right)$ and $\omega^6_{\varphi}\Big(f,\frac{1}{n^{1/3}}\Big)$ we conclude the proof.
\end{proof}
\begin{corollary}\label{lastcor}
 By the equivalence $\omega^6_{\varphi}(f,\frac{1}{n^{1/3}})$ of $\omega^3_{\varphi}(f''',\frac{1}{n^{1/3}})$ and in view of the limit $\lim_{n\rightarrow \infty}\frac{\omega^3_{\varphi}(f''',\frac{1}{n^{1/3}})}{\frac{1}{n^{1/3}}}>0$ it follows that
 \[\|D^{M,2}_n(f,\cdot)-f(\cdot)\|\sim \frac{1}{n}\omega^3_{\varphi}(f''',\frac{1}{n^{1/3}})\sim \frac{1}{n^2}.\]
\end{corollary}
\begin{remark}
The Cor.\ref{lastcor}  verifies the rate of convergence given in Theorem \ref{Vornovaskayatype} and Theorem 3.1 in\cite{ana2019} for the operators $D^{M,2}_n(f,x).$
\end{remark}

\noindent\textbf{Acknowledgement}\\
 The author, Asha Ram Gairola, was supported by project under MHRD/UGC - Empowered Committee's Basic Science Research Programme (No.F.30-371/2017(BSR).

\end{document}